\documentclass[smallabstract,smallcaptions]{dccpaper}
\usepackage[T1]{fontenc}
\usepackage{lmodern}

\usepackage{amsmath, amsthm}
\usepackage{amssymb}
\usepackage{booktabs}
\usepackage{url}
\usepackage{tikz}
\usepackage{colortbl}
\usepackage{comment}
\usepackage{float}
\usepackage[indLines=true,noEnd=true]{algpseudocodex}
\tikzset{algpxIndentLine/.style={draw=black!35,line width=0.4pt}}

\newtheorem{proposition}{Proposition}
\newtheorem{theorem}{Theorem}
\newtheorem{lemma}{Lemma}

\newcommand{\op}[1]{\textsc{#1}}
\newcommand{\sw}[1]{\texttt{#1}}
\newcommand{\dmin}{d_{\min}}

\usepackage{color}

\begin{document}

\title
{\large
\textbf{The Prime Wavelet Tree:\\ Compressing the Carmichael Numbers}}

\author{%
Ankur Gupta$^{\ast}$, Jonathan Webster$^{\dag}$, and Sylvia Webster$^{\ast}$\\[0.5em]
{\small\begin{minipage}{\dimexpr\linewidth-2\tabcolsep\relax}\begin{center}
\begin{tabular}{ccc}
$^{\ast}$Dept.\ of Computer Science and & \hspace*{0.3in} & $^{\dag}$Dept.\ of Mathematical Sciences \\
Software Engineering && Butler University \\
Butler University && Indianapolis, IN, USA \\
Indianapolis, IN, USA && \url{jewebste@butler.edu} \\
\url{agupta@butler.edu} && \\
\url{sjwebster@butler.edu} && \\
\end{tabular}
\end{center}\end{minipage}}
}

\maketitle
\thispagestyle{empty}

\begin{abstract}
We present two number-theoretic compression techniques.
The first is \textit{local}, and compresses individual numbers.
The second is \textit{global}, and compresses the entire set of numbers using a \textit{prime wavelet tree}, which is a data structure of independent interest.
Our techniques apply to any set of numbers satisfying a Korselt-like criterion.
We demonstrate our compression techniques using the recently completed tabulation of all $308{,}279{,}939$ Carmichael numbers less than $10^{24}$ that occupies $18.4$ gigabytes as text.
When combined with standard compression techniques and a heuristic choice of divisors, the resulting file is $588.7$ megabytes ($31.2$ times smaller than the text file).
The initial $80$-bit numbers now use about $15.3$ bits.
\end{abstract}

\section{Introduction}\label{sec:intro}

Computational number theorists continue to produce large data sets for pure mathematicians and other interested researchers.
However, the size of the data can become an obstacle to use in practice.
We focus on the recent tabulation~\cite{shallue-webster2026} of all Carmichael numbers (CNs) less than $10^{24}$, which produced $308{,}279{,}939$ numbers, each listed with its full prime factorization, in an $18.4$ GB text file that researchers must parse.
For example,~\cite{ours-go-to-211} searched for subset products that were also Carmichael numbers.
Running \textsc{gzip -9} on this file achieves a very mild compression ratio of $2.54\times$, since text-based compression cannot exploit the underlying arithmetic structure of CNs.


Our key contributions employ an \sw{ABC} framework, where \sw{A} and \sw{B} represent two novel number-theoretic approaches to compression.
The first, \sw{A}, is local, and defines the quantity $\dmin(n)$, a divisor of $n$ that can be used to recover $n$.
The second, \sw{B}, is global and uses the \textit{prime wavelet tree}, which encodes each number as a path of the primes of one of its divisors.
The third, \sw{C}, is a toolkit of encoding techniques applied judiciously to whatever \sw{A} and \sw{B} store.

The rest of the paper is organized as follows.
In Section~\ref{sec:cn}, we provide required number-theoretic context for CNs. 
This leads naturally to Section~\ref{sec:cnd}, where the local divisor-based representation 
of $n$ is introduced.
In Section~\ref{sec:pwt}, we introduce the prime wavelet tree, which encodes each number by the primes of one of its divisors, as a path in a tree that all numbers share.
Section~\ref{sec:all_methods} briefly reviews standard compression techniques, defines \sw{C}, and measures all eight combinations of \sw{A}, \sw{B}, and \sw{C} on the full table of CNs, together with their decode and factoring times.
Using all three approaches at once produces a $588.7$ megabytes file (about $15.3$ bits per element).
Every CN and its factorization are recovered in $11$ minutes on a single core.


For this paper, all sizes are expressed in exact byte counts from full-table runs (where one gigabyte is $10^9$ bytes).
While our immediate concern is to represent CNs, any tabulation of numbers that satisfies a Korselt-like criterion (see Proposition~\ref{prop:korselt}) could take advantage of the techniques presented here, because they exploit local and global divisibility conditions.
Some explicit examples include the base-$2$ Fermat pseudoprimes \cite{Feitsma13}, absolute Lucas pseudoprimes \cite{helm_webster}, or Lucas-Carmichael numbers \cite{oeis}.

\section{Carmichael numbers}\label{sec:cn}

Fermat's little theorem says that if $n$ is prime, then $b^n \equiv b \pmod{n}$ for every integer $b$.
Its contrapositive gives a fast test that separates most composite numbers from primes.
This test fails for \textit{Carmichael numbers} (CN), composite numbers $n$ such that $b^n \equiv b \pmod{n}$ for every integer $b$.
The smallest CN is $561 = 3 \cdot 11 \cdot 17$.

Carmichael numbers are most frequently understood by their prime factors.
Many comparisons are made between primes and Carmichael numbers, but algorithmically, they are very different problems.
Very few applications would value a table of primes be constructed to share with a wider audience.
In practice, a prime sieve is so efficient that it could generate the primes locally on demand.
This is not true for CNs.

\begin{proposition}[Korselt's criterion]\label{prop:korselt}
$n$ is a Carmichael number if and only if $n$ is composite, square-free, and $(p-1) \mid (n-1)$ for every prime $p \mid n$.
\end{proposition}

We state one key asymptotic result on the count of Carmichael numbers.
\begin{theorem}[PSW-bound \cite{pom_self}]
The count $C(B)$ of Carmichael numbers less than $B$ satisfies
\[ C(B) \leq B \exp \left\{ \frac{-(1 - o(1))\log B \log \log \log B}{ \log \log B }\right\}.\]
We let $\mathcal{E}(B)$ denote this bound.
\end{theorem}

In \cite{pom_self, dist_psp}, Pomerance gives a heuristic argument that $\mathcal{E}(B)$ also provides a lower bound on $C(B)$.
It might be somewhat surprising then that there is a tabulation algorithm whose running time is $\mathcal{E}(B)$~\cite[Theorems 3 and 4]{shallue-webster2026}, thus it is optimal if Pomerance's heuristic lower bound holds.
Even so, the most recent table of CNs, produced by a related algorithm, took about $150$ core-years of compute~\cite{shallue-webster2026}.

We do not want others to redo this tabulation, as a correct but impractical decompression algorithm.
Our goal is to make the table as accessible as possible.
The first major improvement is to avoid duplicating the storage of $n$ and its prime factorization.
It is trivial to reconstruct $n$ from its prime factors and confirm that the constructed number is a CN.
It is also the case that if $n$ is a CN, its prime factorization may be found in polynomial time (see \cite[Section 3]{shallue-webster2026} for precise statements).
For CNs, we avoid the generic integer factorization problem (which is widely regarded as a difficult problem),
and this special-purpose factoring algorithm plays a key role in our decoding algorithms.
Most of our improvements the compression ratio come from identifying and removing redundant divisibility information among these numbers.

\section{Local compression: smaller divisors}\label{sec:cnd}

If we know $p$ divides $n$, we should be able to represent $n$ without duplicating the fact that $(p-1) \mid (n-1)$.
Let $\lambda(n) = \mathrm{lcm}(p_1{-}1, \ldots, p_j{-}1)$, then Korselt's criterion says CNs satisfy $n \equiv 1 \pmod{\lambda(n)}$.
Writing $n = dr$ and using $\lambda(d) \mid \lambda(n)$, we have
\[ n \equiv 1 \pmod{\lambda(n)}  \quad \mbox{ implies } \quad   dr \equiv 1 \pmod{\lambda(d)} .\]
Let $0 \leq r^\star \leq \lambda(d)$ be $d^{-1} \pmod{ \lambda(d)}$, then $n = d( r^\star + k\lambda(d))$ for some integer $k \geq 0$.

Being able to represent $n = d( r^\star + k\lambda(d))$ allows a representation of the form $n = (d,k)$ for any divisor $d > 1$ of $n$. 
For example, $561$ has $7$ representations: $(561, 0)$, $(33, 1)$, $(51,0)$, $(187,0)$, $(3, 93)$, $(11, 5)$, $(17,2)$.
It is straightforward to obtain $n$ from $(d,k)$.
Factor $d$, compute $\lambda(d)$, compute $r^\star$, and finally, compute $n = d( r^\star + k\lambda(d))$.

Experimentation showed that the parameter~$k$ is not very compressible.
We therefore minimize $k$ and choose a representation of $n$ for which $k = 0$.

\begin{proposition}\label{prop:dmin}
Let $n$ be a CN and $d \mid n$ a divisor with $d \cdot \lambda(d) > n$. Then $k = 0$.
\end{proposition}
\begin{proof}
Let $d$ be a divisor of a CN $n$ with $d \cdot\lambda(d) > n$.
Then $n = d\cdot r^\star$.
If $k > 0$, then $n = d(r^\star + k\lambda(d)) > d \cdot\lambda(d) > n $, a contradiction.
\end{proof}

Call a divisor $d$ of $n$ \emph{stopping} if $d \cdot \lambda(d) > n$, and let $\dmin(n)$ be the smallest stopping divisor of $n$.
So, $\dmin(561)= 51$.
We replace the table of $n$ with $\dmin(n)$;
the two sets are in bijection.
The original file had $80$-bit numbers with a median at $78$-bits.
The set of $\dmin(n)$ are $70$-bit numbers with a median of $50$-bits.
On average, we get significant compression.
A few numbers resist this compression.
Consider $n = 109 \cdot 113 \cdot 193 \cdot 337 \cdot 433 \cdot 449 \cdot 577 \cdot 757 \cdot 12{,}097$ with $\lambda(n) = 12{,}096$.
Then, $n$ has $\dmin(n) = n/757$ and this encoding keeps $70$ of the $80$ bits of $n$, matching the worst case.

Representing $\dmin(n)$ in raw binary means that we no longer have easy access to its prime factors.
Decoding requires generic integer factorization, which cannot currently be accomplished in polynomial time.
Once $r^\star$ is obtained we know $dr^\star$ is a CN, and we may use \cite[Algorithm 1]{shallue-webster2025} to factor $r^\star$ in polynomial time.

\section{The prime wavelet tree}\label{sec:pwt}

The wavelet tree of Grossi, Gupta, and Vitter~\cite{ggv2003} represents a sequence over an alphabet $\Sigma$ by a tree whose nodes split the alphabet of their subtree
and store bitvectors recording the part each symbol falls in.
Shaping the tree by the access frequencies cuts its average query time~\cite{ggv2004}, and allowing nodes with more than two children reduces its height~\cite{fmmn2007}.
See Navarro~\cite{navarro2014} for a survey of its myriad uses and variations.

We describe the \emph{prime wavelet tree}, which is a prefix tree over the prime divisors of a stopping number for each CN induced by some fixed ordering of these primes.
As with a wavelet tree or wavelet trie~\cite{go2012} of strings, each level of the tree refines (reduces) the available alphabet of symbols.
In our case, each node stores the primes of its children rather than a bitvector.
The prime wavelet tree is an indexable dictionary~\cite{ghsv2007} that also supports $\op{select}(i)$, $\op{member}(n)$, and $\op{prank}(n)$.
We defer the details of how these functions work, since they have little value for computational number theorists.

\subsection{Definitions}\label{sec:pwt-struct}
Let $\mathcal{C}_B$ be the set of CNs less than $B$.
The tree stores one stopping divisor for each $n \in \mathcal{C}_B$, $\dmin(n)$ unless we say otherwise, and the \emph{prime set} $\pi(n)$ is the set of primes of this stored divisor.
The tree construction fixes a total order $\prec$ on the primes, the \emph{depth order}, and the \emph{path} of $n$ lists $\pi(n)$ in this order.

The \emph{prime wavelet tree} $T = T(\mathcal{C}_B)$ has a node $v_S$ for every set $S$ of the first $h$ primes of some path ($0 \le h \le |\pi(n)|$), which $\prec$ determines as a sequence.
The root is $v_\emptyset$, and the children of $v = v_S$ are the nodes $vq = v_{S \cup \{q\}}$ of $T$ with $q$ after every prime of $S$ in $\prec$, listed in a \emph{sibling order} that the construction also fixes.
The \emph{value-ordered} tree lists each path by decreasing prime but the children of each node by increasing prime.
Figure~\ref{fig:pwt} shows a small example of the tree induced by this ordering.

A set $S$ has the divisor $d_S = \prod_{p \in S} p$ and $\lambda_S = \lambda(d_S) = \mathrm{lcm}\{p - 1 \mid p \in S\}$, with $d_\emptyset = \lambda_\emptyset = 1$.
At the node $v = v_S$ we write these quantities as $d_v$ and $\lambda_v$.
The depth of $v$ is $h = |S|$.
$p_v$ is the prime that $v$ adds to its parent's set, the last prime of $S$ in the ordering $\prec$.\footnote{And $p_v = \bot$ at the root.}
$v_S$ is \emph{terminal} if $S = \pi(n)$ for some $n \in \mathcal{C}_B$.
All leaves must be terminal, and internal nodes may be terminal as well, though Figure~\ref{fig:pwt} does not show any such instances.
The \emph{path order} of $\mathcal{C}_B$ lists its elements by a pre-order traversal of $T$ that visits each node before its children, and the children in sibling order.

\begin{figure}[H]
\begin{center}
\definecolor{termfill}{rgb}{0.98,0.80,0.50}%
\begin{tikzpicture}[x=1mm,y=1mm,
  every node/.style={circle,draw,inner sep=0.4pt,minimum size=4.4mm,font=\scriptsize},
  term/.style={fill=termfill}]
\node (r) at (70.5,6) {$1$};
\node (a17) at (4.5,-9) {17}; \node (a19) at (20.5,-9) {19}; \node (a23) at (32,-9) {23};
\node (a29) at (48,-9) {29}; \node (a31) at (68.5,-9) {31}; \node (a37) at (93.5,-9) {37};
\node (a41) at (127.5,-9) {41};
\node[term] (b1) at (0,-18) {3};   \node[term] (b2) at (9,-18) {5};
\node[term] (b3) at (20.5,-18) {7};  \node[term] (b4) at (32,-18) {7};
\node[term] (b5) at (43.5,-18) {5};  \node (b6) at (52.5,-18) {19};
\node[term] (b7) at (64,-18) {7};  \node[term] (b8) at (73,-18) {17};
\node (b9) at (84.5,-18) {13};       \node (b10) at (93.5,-18) {23};       \node (b11) at (102.5,-18) {31};
\node[term] (b12) at (114,-18) {13}; \node (b13) at (123,-18) {17};    \node (b14) at (132,-18) {29}; \node (b15) at (141,-18) {31};
\node[term] (c6) at (52.5,-27) {5};
\node[term] (c9) at (84.5,-27) {7};  \node[term] (c10) at (93.5,-27) {5};  \node[term] (c11) at (102.5,-27) {19};
\node[term] (c13) at (123,-27) {13}; \node[term] (c14) at (132,-27) {13}; \node[term] (c15) at (141,-27) {11};
\foreach \a in {a17,a19,a23,a29,a31,a37,a41} \draw (r) -- (\a);
\draw (a17)--(b1) (a17)--(b2) (a19)--(b3) (a23)--(b4) (a29)--(b5) (a29)--(b6)
      (a31)--(b7) (a31)--(b8) (a37)--(b9) (a37)--(b10) (a37)--(b11)
      (a41)--(b12) (a41)--(b13) (a41)--(b14) (a41)--(b15);
\draw (b6)--(c6) (b9)--(c9) (b10)--(c10) (b11)--(c11) (b13)--(c13) (b14)--(c14) (b15)--(c15);
\end{tikzpicture}\\[0.4em]
{\scriptsize
\begin{tabular}{@{}lrrrr@{}}
\toprule
path & $d$ & $\lambda(d)$ & $r^\star$ & $n = d\cdot r^\star$ \\
\midrule
$17{\to}3$ & 51 & 16 & 11 & 561 \\
$17{\to}5$ & 85 & 16 & 13 & 1{,}105 \\
$19{\to}7$ & 133 & 18 & 13 & 1{,}729 \\
$23{\to}7$ & 161 & 66 & 41 & 6{,}601 \\
$29{\to}5$ & 145 & 28 & 17 & 2{,}465 \\
$29{\to}19{\to}5$ & 2{,}755 & 252 & 163 & 449{,}065 \\
$31{\to}7$ & 217 & 30 & 13 & 2{,}821 \\
$31{\to}17$ & 527 & 240 & 143 & 75{,}361 \\
\bottomrule
\end{tabular}\hspace{1.5em}
\begin{tabular}{@{}lrrrr@{}}
\toprule
path & $d$ & $\lambda(d)$ & $r^\star$ & $n = d\cdot r^\star$ \\
\midrule
$37{\to}13{\to}7$ & 3{,}367 & 36 & 19 & 63{,}973 \\
$37{\to}23{\to}5$ & 4{,}255 & 396 & 247 & 1{,}050{,}985 \\
$37{\to}31{\to}19$ & 21{,}793 & 180 & 97 & 2{,}113{,}921 \\
$41{\to}13$ & 533 & 120 & 77 & 41{,}041 \\
$41{\to}17{\to}13$ & 9{,}061 & 240 & 61 & 552{,}721 \\
$41{\to}29{\to}13$ & 15{,}457 & 840 & 673 & 10{,}402{,}561 \\
$41{\to}31{\to}11$ & 13{,}981 & 120 & 61 & 852{,}841 \\
\\
\bottomrule
\end{tabular}}
\end{center}
\caption{\label{fig:pwt}%
Value-ordered prime wavelet tree of all 15 CNs $n$ whose prime set $\pi(n)$ has largest prime of at most 41.
An amber node is terminal.}
\end{figure}
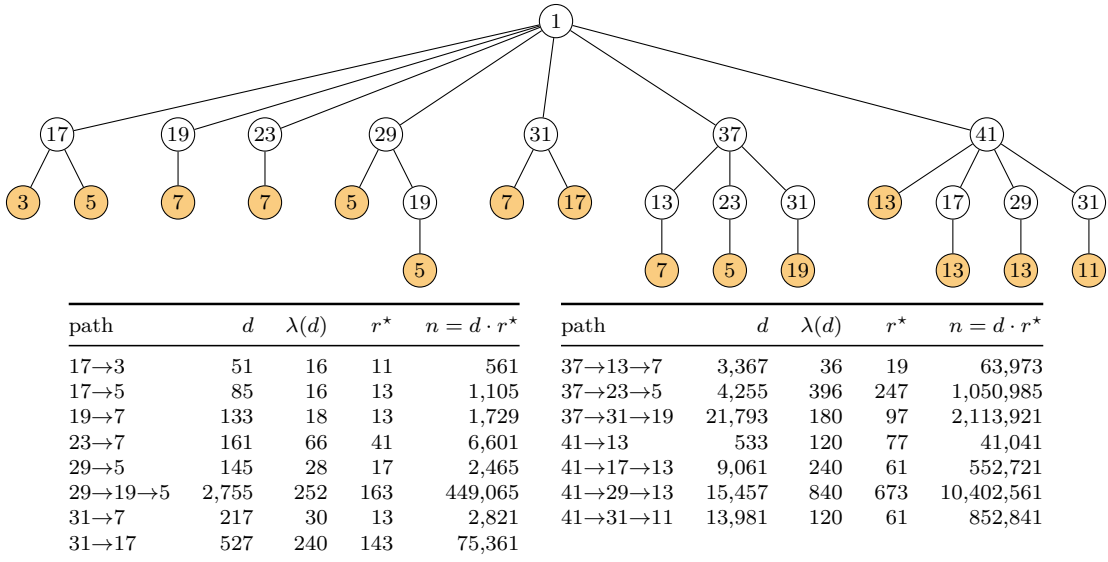

We justify why the word ``wavelet'' applies, in spirit, to this data structure.

\begin{lemma}\label{lem:alpha}
Every child $vq$ of a node $v$ satisfies $q \nmid \lambda_v$ and $\gcd(q - 1, d_v) = 1$.
\end{lemma}
\begin{proof}
Choose $n \in \mathcal{C}_B$ whose path passes through $vq$.
Each prime $p$ of $d_v$ divides $n$, so $(p - 1) \mid (n - 1)$ by Proposition~\ref{prop:korselt}, and therefore $\lambda_v \mid n - 1$.
Since $q \mid n$, $q \nmid (n - 1)$, so $q \nmid \lambda_v$.
Symmetrically, $(q - 1) \mid (n - 1)$ and $d_v \mid n$, so $\gcd(q - 1, d_v)$ divides $\gcd(n - 1, n) = 1$.
\end{proof}

Each node draws its children from a smaller alphabet than its parent does, as in a wavelet tree.
The child primes of $v$ are among the primes that Lemma~\ref{lem:alpha} does not exclude and that follow $p_v$ in $\prec$ (at the root, any prime; in the value-ordered tree, a prime below $p_v$).
For example, the root child 31 in Figure~\ref{fig:pwt} has $\lambda_v = 30$, so its children cannot include 3 or 5.
Our coders do not use the exclusions of Lemma~\ref{lem:alpha}.
Lemma~\ref{lem:decode} shows that the prime wavelet tree~$T$ correctly represents all CN in $\mathcal{C}_B$.

\begin{lemma}\label{lem:decode}
For every $n \in \mathcal{C}_B$, the terminal node $v = v_{\pi(n)}$ determines $n = d_v \cdot r_v^\star$, where $r_v^\star = d_v^{-1} \bmod \lambda_v$.
Hence $n \mapsto v_{\pi(n)}$ is a bijection from $\mathcal{C}_B$ onto the terminal nodes of $T$.
\end{lemma}
\begin{proof}
The divisor $d_v$ stored for $n$ is stopping, so Proposition~\ref{prop:dmin} gives $n = d_v \cdot r_v^\star$.
The map is onto the terminal nodes by their definition, and it is one-to-one because $v_{\pi(n)} = v_{\pi(m)}$ gives $n = d_v r_v^\star = m$.
\end{proof}

When the tree stores $\dmin(n)$, every proper subset $S$ of $\pi(n)$ gives a divisor $d_S < \dmin(n)$ of $n$, which is not stopping, so in every order the path of $n$ ends at the first node whose divisor is stopping for $n$.
The primes of $r^\star$ are not on this path, so the decoder recovers only their product $r_v^\star$.
The stopping rule fails when the stored divisor is $n$ itself or is chosen heuristically (as we do in Section~\ref{sec:results}), since a proper subset of $\pi(n)$ can then give a stopping divisor.
Since stopping depends on $n$, a node can be terminal for one number and internal for a larger one, so $T$ marks its terminal nodes explicitly.

Unlike a wavelet tree, which represents a sequence, the prime wavelet tree represents the set $\mathcal{C}_B$.
A wavelet tree over $\mathcal{C}_B$ in numeric order would also record which child each element enters and would require an additional $\log_2 N!$ bits of storage.
The prime wavelet tree lists $\mathcal{C}_B$ in path order instead, and keeps only the shape of the tree, its terminal marks, and the primes at its nodes.
This also is a source of our compression, since we store a prefix shared by many paths only once.

\subsection{Encoding}\label{sec:pwt-enc}

The encoding $\mathrm{PWT}(T)$ lists $T$ in pre-order.
Each node writes its number of children and, if it is internal and not the root, whether it is terminal.
Each child then follows as the gap from its previous sibling, and then its own subtree.
In the value-ordered tree, gaps at depth 1 count odd integers, and deeper gaps count odd primes, which the decoder unranks with a bitvector of the primes below $B^{1/3} = 10^8$ for our tabulation.
This suffices because only the largest prime of $\dmin(n)$ can exceed $n^{1/3}$: if $p > q$ are its two largest primes, then $\lambda(n/q) \ge p - 1 > q$, so $n/q$ is stopping, hence $n/q \ge \dmin(n) \ge pq$ and $q^3 < pq^2 \le n < B$.

Figure~\ref{fig:algo} provides pseudocode for the encoding and decoding of $\mathrm{PWT}(T)$.

\begin{theorem}\label{thm:pwt}
$\mathrm{PWT}(T)$ with $t$ nodes decodes to $\mathcal{C}_B$ in polynomial time per node.  
\end{theorem}
\begin{proof}
One pre-order pass, visits children in sibling order and unranks each node's prime $q$.
We can then write $d_{vq} = d_v q$ via standard multiplication. 
The Euclidean algorithm computes $\lambda_{vq} = \mathrm{lcm}(\lambda_v, q - 1)$ and $r^\star \equiv d_v^{-1} \bmod{\lambda_v}$.  
Another multiplication computes $n = d_v \cdot r^\star$.
If the prime factors of $r^\star$ are needed, these can also obtained in polynomial time by Algorithm 1 of \cite{shallue-webster2026}.
\end{proof}

\begin{figure}[t]
{\footnotesize
\algrenewcommand\algorithmicindent{0.8em}%
\algrenewcommand\alglinenumber[1]{{\scriptsize #1:}}%
\newcommand{\ctxa}{{\scriptstyle h,\,\mathrm{bits}(p_v)}}%
\newcommand{\ctxg}{{\scriptstyle h{+}1,\,\mathrm{bk}(c),\,\mathrm{bits}(p_v)}}%
\noindent\rule{\linewidth}{0.8pt}\par
\noindent\makebox[0.492\linewidth][l]{\textbf{Algorithm 1}\enspace\textsc{PWT-Encode}$(\mathcal{C}_B, B^{1/3})$}\hfill
\makebox[0.492\linewidth][l]{\textbf{Algorithm 2}\enspace\textsc{PWT-Decode}$(\mathrm{PWT}(T), B^{1/3})$}\par\vspace{1pt}
\noindent\rule{\linewidth}{0.4pt}\par\vspace{-0.3em}
\noindent\hbox to \linewidth{%
\begin{minipage}[t]{0.492\linewidth}
\begin{algorithmic}[1]
\State $\pi(n) \gets$ the primes of $\dmin(n)$
\State list each $\pi(n)$ decreasing; sort; build $T$
\State sieve the odd primes $< B^{1/3}$ into a bitmap
\State $\mathrm{idx}(q) \gets$ number of odd primes $\le q$
\State \Call{EmitNode}{$v_\emptyset$, $0$, $\bot$}
\Procedure{EmitNode}{$v$, $h$, $p_v$}
  \State $c \gets$ number of children of $v$
  \State $\textsc{Emit}_{\ctxa}(c)$
  \State \textbf{if} $c, h > 0$ \textbf{then} $\textsc{Emit}_{\scriptstyle h}(\mathrm{terminal}(v))$
  \State $\mathrm{prev} \gets 1$ \textbf{if} $h = 0$ \textbf{else} $0$
  \ForAll{child primes $q$ of $v$, increasing}
    \State \textbf{if} $h = 0$ \textbf{then} $g \gets (q - \mathrm{prev})/2$;\ $\mathrm{prev} \gets q$
    \State \textbf{else} $g \gets \mathrm{idx}(q) - \mathrm{prev}$;\ $\mathrm{prev} \gets \mathrm{idx}(q)$
    \State $\textsc{EmitLenRaw}_{\ctxg}(g)$
    \State \Call{EmitNode}{$vq$, $h{+}1$, $q$}
  \EndFor
\EndProcedure
\end{algorithmic}
\end{minipage}\hfill{\color{black!25}\vrule}\hfill
\begin{minipage}[t]{0.492\linewidth}
\begin{algorithmic}[1]
\State sieve the odd primes $< B^{1/3}$ into a bitmap
\State $\mathrm{sel}(j) \gets$ the $j$-th odd prime
\State \Call{DecodeNode}{$0$, $\bot$, $1$, $1$}
\Procedure{DecodeNode}{$h$, $p_v$, $d$, $\lambda$}
  \State $c \gets \textsc{Read}_{\ctxa}()$
  \State $\tau \gets \textsc{Read}_{\scriptstyle h}()$ \textbf{if} $c, h > 0$ \textbf{else} $(c = 0)$
  \State \textbf{if} $\tau$ \textbf{then} output $n \gets d \cdot (d^{-1} \bmod \lambda)$
  \State $\mathrm{prev} \gets 1$ \textbf{if} $h = 0$ \textbf{else} $0$
  \For{$i \gets 1$ \textbf{to} $c$}
    \State $g \gets \textsc{ReadLenRaw}_{\ctxg}()$
    \State \textbf{if} $h = 0$ \textbf{then} $q \gets \mathrm{prev} + 2g$;\ $\mathrm{prev} \gets q$
    \State \textbf{else} $\mathrm{prev} \gets \mathrm{prev} + g$;\ $q \gets \mathrm{sel}(\mathrm{prev})$
    \State \Call{DecodeNode}{$h{+}1,\, q,\, dq,\, \mathrm{lcm}(\lambda, q{-}1)$}
  \EndFor
\EndProcedure
\end{algorithmic}
\end{minipage}}\par\vspace{7pt}
\noindent\rule{\linewidth}{0.8pt}}
\caption{\label{fig:algo}%
Encoder and decoder of the value-ordered prime wavelet tree.
Here $d = d_v$ and $\lambda = \lambda_v$; subscripts are rANS contexts; \textsc{EmitLenRaw}/\textsc{ReadLenRaw} code a gap as its bit length and raw bits; counts $c \ge 4095$ escape to an Elias gamma code~\cite{elias1975} of $c - 4094$; $\mathrm{bits}(\cdot)$ is a bit length ($\mathrm{bits}(\bot) = 0$); and $\mathrm{bk}(c) \in \{1, 2-3, 4-15, 16{+}\}$ buckets the child count.}
\end{figure}

\subsection{The frequency-ordered prime wavelet tree}\label{sec:pwt-freq}

As we have mentioned earlier, wavelet trees can be shaped (in a Huffman ordering, for example) to 
reduce its overall query time.
In the same vein, we put frequent primes near the root of the prime wavelet tree to reduce the total number of nodes.
Hence, we can group together frequent small primes across many subtrees and encode them once.

The \emph{frequency-ordered} version of a prime wavelet tree ranks the primes by their \emph{path count}, the number of paths containing them, in decreasing order, with ties going to the larger prime.
Its depth and sibling orders are both increasing rank, similar to the FP-tree~\cite{hpy2000}, so every child ranks after its parent.
Lemmas~\ref{lem:alpha} and~\ref{lem:decode} are true for every ordering~$\prec$.
The prime of rank~$k$ labels at most $2^{k-1}$ nodes, one per subset of the $k-1$ primes ranked before it.
On the full table, $\dmin$ shrinks the total number of tree nodes by 41\%.

Children are now sorted by rank, so $\mathrm{PWT}(T)$ counts gaps over ranks at every depth, measuring a first child's gap from its parent's rank.
The decoder still needs the primes to form $d_v$ and $\lambda_v$ and reads them from a \emph{rank table}.
Within a \emph{count class}, a maximal run of ranks with equal path count, the tie rule sorts the primes, so the table stores the class sizes (not the counts) and each class's primes in increasing order as gaps, coded like tree gaps.

\section{Results}\label{sec:all_methods}\label{sec:results}

We measure all eight combinations of the three switches (\sw{A}, \sw{B}, and \sw{C}) in Table~\ref{tab:abc}.
For scale, the text file takes 476.4 bits per CN and \textsc{gzip -9} takes 187 bits per CN.
A CN in fixed width takes 80, and encoding an arbitrary set of $N$ integers below $10^{24}$ needs $\lceil \log_2\binom{10^{24}}{N} \rceil$ bits overall, or 52.97 per element.


Switch \sw{A} replaces each $n$ by $\dmin(n)$. 
Switch \sw{B} stores each number's prime set $\pi(n)$ as a path in the frequency-ordered prime wavelet tree. 
The path will hold the primes of $\dmin(n)$ when \sw{A} is on and all primes of $n$ when it is off.
Switch \sw{C} applies the coding techniques from Table~\ref{tab:tricks}.
Without it, every stored integer is a single VByte code.
All trees in Table~\ref{tab:abc} are frequency-ordered
Cells \sw{101}, \sw{010}, \sw{110}, \sw{011}, and \sw{111}, and the PWT + $d^{\ast}$ row are stored files.
The other three cells were encoded and decoded in memory.


When \sw{C}~$=0$, every stored integer is one VByte code.
A list stores the gaps between consecutive values, and a tree stores one VByte per node (twice its child count, plus its terminal flag), one per child (its rank gap), and its rank table as VByte gaps between the primes of each count class.
With \sw{C}~$=1$, these codes are replaced by the tricks in Table~\ref{tab:tricks}, each built into a decodable file.

\begin{table}[H]
\begin{center}
\caption{\label{tab:tricks}%
Savings of each bit-reduction trick in bits per CN.}
\vspace{4pt}
{\renewcommand{\baselinestretch}{1}\small\setlength{\tabcolsep}{4pt}%
\begin{tabular}{@{}lllrr@{}}
\toprule
 & trick & exploits & \sw{A}~$=0$ & \sw{A}~$=1$ \\
\midrule
\emph{lists} & interpolative coding & clustering of the values & 8.91 & 4.84 \\
 & parity & every value is odd & 1.00 & 1.00 \\
 & residue classes mod $M$ & non-uniform residues (Korselt) & 11.31 & 1.01 \\
\midrule
\emph{trees} & rANS, length $+$ raw bits & skewed counts, flags, gap lengths & 32.15 & 15.90 \\
 & richer contexts & prefix product, siblings, expected gap & 2.14 & 1.90 \\
 & gap-mantissa models & likely next rank after a frequent one & 0.92 & 0.15 \\
\bottomrule
\end{tabular}}
\end{center}
\end{table}

\emph{Lists.}
A list is split into residue classes modulo $M$, and class $r$ stores the quotients $(v-r)/M$ of its values $v$ by binary interpolative coding~\cite{moffat-stuiver2000}.
The split requires only a directory of class sizes and end values, since a decoded value determines its own class.
The best of 24 different moduli we tried is $M = 10{,}810{,}800 
$  for $n$ and $M = 2310 
$ for $\dmin$.

\emph{Trees.}
Each gap is split into its bit length and its \emph{mantissa}, the bits below its leading one.
A static rANS coder~\cite{duda2013,moffat-petri2020} codes the counts, terminal flags, and bit lengths, with one frequency table per context.
Each gap is measured from its \emph{reference rank}, which is either the previous sibling's rank or the parent's for a first child.
The basic contexts are the depth, the child count (in buckets), and the bit length of the reference rank.
Richer contexts also extend these with the bit length of the prefix product $d_v$, whether the child is first, the number of siblings left, and the expected gap (which is the bit length of the unused rank range per remaining child).
Mantissas are written raw unless a gap-mantissa model encodes them.
For gaps longer than 12 bits, the gap-mantissa model only writes the top four bits.
Each model is keyed by the reference rank if it is the root's or among the 1{,}023 most frequent, and by a coarser bucket otherwise.

\begin{table}[H]
\begin{center}
\caption{\label{tab:abc}%
The \sw{A}/\sw{B}/\sw{C} matrix on the full table of $308{,}279{,}939$ CNs.
Times are single-threaded minutes on an Intel Core i9-9900K.
\emph{Build} is the construction time of its encoding.
\emph{Decode} recovers every CN, alone or with its factorization.
The lists of $n$ decode in sorted order.
Sorting the others costs 0.8 minutes more.
$^\dagger$The $\dmin$ list decodes in 0.03 and 0.61 minutes but the bulk of the time comes from factoring every $d$.
$^\ddagger$The heuristic model~$d^{\ast}$ in place of $\dmin$.
$^\S$Wall clock time for 15 threads.}
\vspace{4pt}
{\renewcommand{\baselinestretch}{1}\renewcommand{\arraystretch}{1.35}\small
\setlength{\tabcolsep}{2.1pt}\setlength{\aboverulesep}{0pt}\setlength{\belowrulesep}{0pt}%
\definecolor{bitsbg}{rgb}{0.80,0.88,0.97}\definecolor{bitsdk}{rgb}{0.55,0.71,0.91}%
\definecolor{bldbg}{rgb}{0.86,0.95,0.84}\definecolor{blddk}{rgb}{0.62,0.84,0.60}%
\definecolor{decbg}{rgb}{1.00,0.96,0.85}\definecolor{decdk}{rgb}{0.98,0.80,0.50}%
\newcommand{\unit}[2]{{\renewcommand{\arraystretch}{1}\begin{tabular}[c]{@{}c@{}}\textbf{#1}\\(#2)\end{tabular}}}%
\newcommand{\two}[2]{{\renewcommand{\arraystretch}{1}\begin{tabular}[c]{@{}c@{}}\textbf{#1}\\\textbf{#2}\end{tabular}}}%
\newcommand{\hdS}[1]{\multicolumn{1}{>{\columncolor{bitsbg}}c}{#1}}
\newcommand{\hdB}[1]{\multicolumn{1}{>{\columncolor{bldbg}}c}{#1}}
\newcommand{\hdD}[1]{\multicolumn{1}{>{\columncolor{decbg}}c}{#1}}
\newcommand{\hdDr}[1]{\multicolumn{1}{>{\columncolor{decbg}}c|}{#1}}
\newcommand{\blank}{\multicolumn{1}{c}{}}
\newcommand{\blankr}{\multicolumn{1}{c|}{}}
\newcommand{\z}{\phantom{0}}%
\newcommand{\dg}{\rlap{$^\dagger$}\z}
\begin{tabular}{l|c>{\columncolor{bitsbg}}c>{\columncolor{bldbg}}r>{\columncolor{decbg}}r>{\columncolor{decbg}}r|c>{\columncolor{bitsbg}}c>{\columncolor{bldbg}}r>{\columncolor{decbg}}r>{\columncolor{decbg}}r}
\toprule
 & & \hdS{\unit{Space}{bits}} & \hdB{\unit{Build}{min}} & \multicolumn{2}{>{\columncolor{decbg}}c|}{\unit{Decode}{min}}
 & & \hdS{\unit{Space}{bits}} & \hdB{\unit{Build}{min}} & \multicolumn{2}{>{\columncolor{decbg}}c}{\unit{Decode}{min}} \\
\multicolumn{1}{c|}{} & \textbf{\sw{ABC}} & \hdS{\two{per}{CN}} & \hdB{\two{oracle}{to file}} & \hdD{\two{CNs}{only}} & \hdDr{\two{CNs +}{primes}}
                      & \textbf{\sw{ABC}} & \hdS{\two{per}{CN}} & \hdB{\two{oracle}{to file}} & \hdD{\two{CNs}{only}} & \hdD{\two{CNs +}{primes}} \\
\midrule
raw $n$       & \sw{000} & 61.31 & \textbf{0.43} & \textbf{0.05} & 33.4\z        & \sw{001} & 40.10 & 1.65    & 0.76     & 34.1\z \\
$\dmin$       & \sw{100} & 31.67 & 158.0\z       & 214.5\dg      & 223.4\z       & \sw{101} & 24.82 & 159.1\z & 215.1\dg & 224.0\z \\
PWT           & \sw{010} & 67.24 & 7.75          & 0.30          & \textbf{0.30} & \sw{011} & 32.03 & 12.7\z  & 2.81     & 2.81 \\
PWT + $\dmin$ & \sw{110} & 33.65 & 167.1\z       & 1.31          & 10.1\z        & \sw{111} & 15.70 & 168.5\z & 2.15     & 11.0\z \\
\midrule
PWT + $d^{\ast}$ & \blank & \blank & \blank & \blank & \blankr & \sw{111}$^\ddagger$ & \cellcolor{bitsdk}\textbf{15.28} & \cellcolor{blddk}2011\rlap{$^\S$}\phantom{.00} & \cellcolor{decdk}2.12 & \cellcolor{decdk}10.8\z \\
\bottomrule
\end{tabular}}
\end{center}
\end{table}

We highlight some of the interesting features that can be seen in the table above.
\begin{itemize}
\item \sw{A} saves space but costs factoring time.  The space usage in \sw{10X} compared to \sw{00X} is significantly improved.   However, decoding \sw{10X} uses generic factorization algorithms on $\dmin$, so the cost is no longer polynomial time per element.  
\item We can see the effect of polynomial time factoring of CNs by looking at the ``only'' column versus ``+ primes'' columns.  The difference is relatively small and the largest difference is seen in \sw{00X} where $n$ itself is factored.  In the \sw{10X} cases, we can see that $\dmin$ is slowly factored and then a polynomial time computation recovers the prime factors of $r^\star$.  And in the \sw{11X} cases, $\dmin$ comes in factored form, so the the cost to build $\lambda(\dmin)$ is comparable to the cost of factoring $r^\star$.  
\item The heuristic model trades build time for space.  Compare the two lines for \sw{111} and see that the dynamic recomputation of stopping numbers is computationally expensive but gains about $0.4$ bits per number.  
\end{itemize}

\section{Conclusion}\label{sec:conclusion}

Our smallest encoding stores all $308{,}279{,}939$ CNs below $10^{24}$ in $588.7$ megabytes, or 15.28 bits per CN and can recover its full factorization in 11 minutes on a single core.
If we are not willing to pay the high build cost to compute the set of $d^\ast$ divisors, we could build it in under 3 hours with a nominal space increase, or in about double the space in under 13 minutes.

Several problems remain open.
The first is how best to choose the divisors.
Picking the best stopping number so that the prefix tree of the chosen sets, under one common order, has the fewest nodes is NP-hard in general, by a reduction from vertex cover~\cite{karp1972}. Each edge $xy$ of a graph becomes an element whose candidates are $\{x, \ell_{xy}\}$ and $\{y, \ell_{xy}\}$.
Second, the tree itself has room to shrink.
The frequency-ordered tree over $\dmin$ requires 1.63 nodes per CN, compared to a floor of 1.



\subsection*{Acknowledgments}
The software implementations in this paper were done by Claude and verified by us.
We used Claude as a sounding board as we wrote this paper, but the words were written by us.
Tables, pseudocode, and figures were generated by Claude.
The authors have vetted all results and we take full responsibility for the content of this paper.

\Section{References}

\end{document}